\documentclass{article}
\usepackage{amsmath, amsthm,tikz}
\newtheorem{theorem}{Theorem}
\newtheorem{lemma}{Lemma}
\newtheorem{corollary}{Corollary}
\newtheorem{remark}{Remark}
\title{An Upper Bound on Burning Number of Graphs}
\author{
Max Land \thanks{Dutch Fork High School, Irmo, SC 29063,
({\tt max.ruikang.land@gmail.com}).}
\and 
Linyuan Lu
\thanks{University of South Carolina, Columbia, SC 29208,
({\tt lu@math.sc.edu}). This author was supported in part by NSF
grant DMS 1300547.}
}

\begin{document}
\maketitle
\begin{abstract}
  The burning number $b(G)$ of a graph $G$ was introduced by Bonato, Janssen, and Roshanbin [Lecture Notes in Computer Science 8882 (2014)] for measuring the speed of the spread of contagion in a graph. They proved for any connected graph $G$ of order $n$,
$b(G)\leq 2\lceil \sqrt{n} \rceil-1$, and 
conjectured  that $b(G)\leq \lceil \sqrt{n} \rceil$. In this paper, we proved  $b(G)\leq \lceil\frac{-3+\sqrt{24n+33}}{4}\rceil$, which is roughly $\frac{\sqrt{6}}{2}\sqrt{n}$. We also settled the following conjecture of Bonato-Janssen-Roshanbin: $b(G)b(\bar G)\leq n+4$ provided both $G$ and $\bar G$ are connected.
\end{abstract}
\section{Introduction}
The burning number of a graph was introduced by
Bonato-Janssen-Roshanbin \cite{burn1,burn2, Ros}. It is related to the contact
processes on graphs such as the Firefighter problem \cite{BDGS,FKMR,FM}.
In the paper \cite{burn1, burn2}, Bonato-Janssen-Roshanbin considered a graph
process which they called {\em burning}. At the beginning of the
process, all vertices are {\em unburned}. During each round, one may
choose an unburned vertex and change its status to {\em burned}. At
the same time, each of the vertices that are already burned, will
remain burned and spread to all of its neighbors and change their
status to burned. A graph is called {\em $k$-burnable} if it can be
burned in at most $k$ steps. The {\em burning number} of a graph $G$, denoted by
$b(G)$, is the minimum number of rounds necessary to burn all vertices
of the graph. For example, $b(K_n)=2$, $b(P_4)=2$, and $b(C_5)=3$. In
the paper \cite{burn2}, they proved $b(P_n)= \lceil n^{1/2}\rceil$. 
Based on this result,  Bonato-Janssen-Roshanbin \cite{burn2}
  made the following conjecture.

{\bf Conjecture 1:} {\it for any connected
graph $G$ of order $n$, $b(G) \leq \lceil n^{1/2}\rceil$.}

Bonato-Janssen-Roshanbin \cite{burn1, burn2} proved
 $b(G)\leq 2\lceil n^{1/2}\rceil-1$. 
The previous best known bound is due to Bonato et al. \cite{burn4}:
$$b(G)\leq \left(\sqrt{\frac{32}{19}}+o(1)\right )\sqrt{n}.$$
In this paper, we improved the upper bound of $b(G)$ as follows.
\begin{theorem}\label{t1}
  If $G$ is a connected graph of order $n$, then $$b(G)\leq
\left \lceil\frac{-3+\sqrt{24n+33}}{4}\right
\rceil.$$
\end{theorem}

In the paper \cite{burn2}, Bonato, Janssen, and Roshanbin
also considered Nordhaus-Gaddum Type
problem on the burning number. Let $\bar G$ be the complement graph
of the graph $G$. In \cite{burn2}, they proved $ b(G) + b(\bar G)\leq n + 2$ and $b(G)
b(\bar G)\leq 2n$. Both bounds are tight and are achieved by the
complete graph and its complement. When both graphs $G$ and $\bar G$
are connected, they proved $b(G) + b(\bar
G)\leq 3\lceil n^{1/2}\rceil -1$ and $b(G) b(\bar
G)\leq n+6$ for all graph $G_n$ of order $n\geq 6$. 
The following conjecture has been made in \cite{burn2}:

{\bf Conjecture 2:} {\it If both $G$ and $\bar G$ are connected graphs of order $n$,
then $b(G)b(\bar G)\leq n+4$.}

Using Theorem \ref{t1},  we settled this conjecture positively.

\begin{theorem}\label{t2} 
If both $G$ and $\bar G$ are connected graphs of order $n$, then $$b(G) b(\bar
G)\leq n+4.$$
The equality holds if and only if $G=C_5$.
\end{theorem}

\section{Notations and Lemmas}
For each positive integer $k$, let $[k]$ denote the set $\{1,2,\ldots, k\}$.
A graph $G=(V,E)$ consists of a set of vertices $V$ and edges $E$.
The {\em order} of $G$, dented by $|G|$, is the number of vertices in $G$.
A graph $G$ is called {\em connected} if for any two vertices 
there is a path connecting them. In this paper, we always assume
that $G$ is a connected graph.
%Let $|G|$ be the order of the graph.
The {\em distance} between two vertices $u$ and $v$, denoted by $d(u,v)$,
 is the length of the shortest
path from $u$ to $v$ in graph $G$. The {\em eccentricity} 
of a vertex $v$ is the maximum distance between $v$ and any other vertex in $G$.
The maximum eccentricity is the {\em diameter} $D(G)$ while the minimum eccentricity is
the {\em radius} $r(G)$. The {\em center} of $G$ is the set of vertices of 
eccentricity equal to the radius.

For any nonnegative integer $k$ and a vertex $u$, the {\em $k$-th closed
neighborhood}
 of $u$ is the set of vertices whose distance from $u$ is
at most $k$; denoted by $N_k[u]$. From the definition, a graph $G$ is 
$k$-burnable if there is a {\em burning sequence} $v_1,\ldots, v_k$ of vertices
such that
\begin{align}
  \label{eq:1}
  &V\subset \cup_{i=1}^kN_{k-i}[v_{i}] \\
\label{eq:2}
 & \forall i,j \in [k]\colon d(x_i,x_j)\geq j-i.
\end{align}
The burning number $b(G)$ is the smallest integer $k$ such that $G$ is $k$-burnable.
It has been shown that Condition \eqref{eq:2} is redundant for the definition of
burning number $b(G)$ (see Lemma 1 of \cite{burn4}). It is often convenient to
rewrite Condition \eqref{eq:1} by relabeling the vertices in the burning sequence
as follows:
\begin{align}
  \label{eq:3}
V\subset \cup_{i=1}^kN_{i-1}[v_{i}].
\end{align}

This leads the following generalization, which 
is very useful for the purpose of induction.
 For a set (or multiset) $A$ of $k$ positive integers
$a_1,a_2,\ldots,a_k$ (not necessarily all distinct),
we say a graph $G$ is {\em $A$-burnable},
if there exist $k$ vertices $v_1,v_2,\ldots,v_k$ such
that $G\subseteq \cup_{i=1}^kN_{a_i-1}[v_{i}].$ 
Under this terminology, the burning number $b(G)$ is the least $k$ so
that $G$ is $[k]$-burnable. 

A {\em tree} is an acyclic connected graph. For any tree $T$, it is well-known
that the center of $T$ consists of either one vertex or two vertices. If the center of 
$T$ consists of one vertex, then
 $D(T)=2r(T)$; otherwise, $D(T)=2r(T)-1$. (See \cite{BLW}.)

A {\em rooted tree} is a tree with one vertex $r$ designated as the {\em root}.
The {\em height} of a rooted tree is the eccentricity of the root.
In a rooted tree, the {\em parent} of a vertex is the vertex connected to it on the path to the root. A {\em child} of a vertex $v$ is a vertex of which $v$ is the parent. 
A {\em descendent} of any vertex $v$ is any vertex which is either the child of $v$ or is (recursively) the descendent of any of the children of $v$. A {\em leaf} vertex
is a vertex with degree $1$ but not equal to the root. The {\em subtree rooted at
$v$} is the induced subgraph on the set of $v$ and its all descendents.
The important observation is that if a subtree rooted at $v$ is pruned from the whole
tree, the remaining part (if non-empty) is still a tree. This observation is very useful
for induction. 

A {\em spanning} tree of graph $G$ is a subtree of $G$ that covers all
vertices of $G$. In the papers \cite {burn1, burn2}, Bonato, Janssen, and Roshanbin
 proved
\begin{equation}
  \label{eq:3}
  b(G)=\min\{b(T): T \mbox { is a spanning subtree of } G\}.
\end{equation}
Thus, it is sufficient to only consider the burning
number $b(T)$ for tree $T$.

First we prove a simple lemma, which illustrates the idea of the induction.
\begin{lemma}\label{l1}
  Let $A=\{a_1,a_2,\ldots,a_k\}$ be a  set of $k$ nonnegative
  integers. If a tree $T$ has order at most $\sum_{i=1}^k a_i +\max\{a_i\colon
1\leq i\leq k\} -1$,
  then $T$ is $A$-burnable.
\end{lemma}
\begin{proof}
With loss of generality, we can assume that $a_1\geq a_2\geq \cdots \geq a_k$.
  We will use induction on $k$. Initial case: $k=1$, $A=\{a_1\}$. 
We need to prove that if a  tree $T$ has at most $2a_1-1$ vertices, then $T$ is
   $A$-burnable. Note that 
$$r(T)\leq \frac{D(T)+1}{2}\leq \frac{n}{2}\leq a_1-\frac{1}{2}.$$
Since the radius $r(T)$ is an integer, we must have $r(T)\leq a_1-1$.
Thus $T$ is $\{a_1\}$-burnable. 

% Consider a longest path $xy$ in $T$. Pick a $v_1$ in
%   the middle of the $xy$ path (break arbitrarily if tie). 
%   Notice the $d(x,v_1)\leq a_1-1$ and $d(y,v_1)\leq
%   a_1-1$. Otherwise, the number of vertices on the $xy$ path exceeds
%   $2a_1-1$. Contradiction.
%  Now we show $T\subseteq N_{a_1-1}[v_1]$. If not, there exists a vertex
%  $z$ that is not in  $T\subseteq N_{a_1}[v_1]$. Then, $d(z,v_1)\geq
%  a_1$. Let $u$ be the most closest to $z$ on the path $xy$. Since $T$ is a tree, 
%  $u$ is either on the subpath $xv_1$  or on the subpath $v_1y$. Without loss of generality, we can assume $u$ is on the subpath $xv_1$. We have
% $$d(z,y)=d(z,u)+d(u,y)=d(z,v_1)+d(v_1,y)>d(x,v_1)+d(v_1,y)=d(x,y).$$
% Contradiction to the assumption that $xy$ is the longest
%  path. 

Now we assume the statement holds for any set of $k-1$ integers. For
any $A$ of $k$ integers $a_1\geq a_2\geq \cdots \geq a_k>0$
 and any tree $T$ with at most $2a_1+a_2+\cdots+a_{k}-1$,
  we will prove that $T$ is $A$-burnable. Pick an arbitrary vertex $r$ as
  the root of $T$.
Let $h$ be the height of this rooted tree. If
$h\leq a_1-1$, then $V(T)\subset N_{a_1-1}(r)$. I.e., $T$ is $\{a_1\}$-burnable.
Thus $T$ is $A$-burnable.

Now we assume $h\geq a_1$.  Select a leaf vertex $u$ such that $d(r,u)=h$.
Let $v_k$ be the vertex on the $ru$-path such
  that the distance $d(u,v_k)=a_k-1$. (This is possible since $h\geq a_1>a_k-1$.
 Let $T_1$ be the subtree rooted
  at $v_k$, and $T_2:=T\setminus T_1$ be the remaining subtree. Notice that $|T_1|\geq
  a_k$. Thus,
  \begin{align*}
    |T_2|&= |T|-|T_1|\\
&\leq 2a_1+a_2+\cdots+a_{k}-1-a_k\\
&=  2a_1+a_2+\cdots+a_{k-1}-1.
  \end{align*}
By inductive hypothesis, $T_2$ is $\{a_1,a_2,\ldots,a_{k-1}\}$-burnable.
 Thus, there exists $k-1$ vertices $v_1,v_2,\ldots,v_{k-1}$ such
that $T_2\subseteq \cup_{i=1}^{k-1}N_{a_i-1}[v_{i}].$
 Also, notice $T_1\subseteq N_{a_k-1}[v_{k}]$. Therefore, $T\subseteq
 \cup_{i=1}^{k-1}N_{a_i-1}[v_{i}].$ The proof of the lemma is finished.
\end{proof}

\begin{remark}
  The bound in Lemma \ref{l1} is tight. 
\end{remark}
\begin{proof}
  Consider the following example: for any positive integer $a$, let $a_1=a_2=\cdots=a_k=a$,  i.e. $A$ is a multiset consisting of $k$ $a$'s. 
Now we will construct a tree $T$ as following. First construct $k+1$
disjoint paths $P_0,P_1,\ldots,P_k$ with each of order $a$. Create
tree $T$ by connecting one endpoint of $P_1,P_2,\ldots,P_k$ to the
same endpoint of $P_0$ (see figure below).

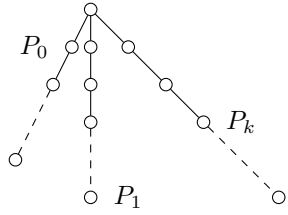
\begin{center}
  \begin{tikzpicture}[scale= 0.5, vertex/.style={circle,scale=0.5,
 draw=black,
      fill=white}]
    \node at (0,0)[vertex] (v0){};
 \node at (-0.5,-1)[vertex] (v1){};
 \node at (-1,-2)[vertex] (v2){};
 \node at (-2,-
4)[vertex] (v3){}; 
\draw (v0)--(v1)--(v2);
\draw [dashed] (v2)--(v3);
 \node at (0,-1)[vertex] (v4){};
 \node at (0,-2)[vertex] (v5){};
 \node at (0,-3)[vertex] (v6){};
 \node at (0,-5)[vertex] (v7){}; 
\draw (v0)--(v4)--(v5)--(v6);
\draw [dashed] (v6)--(v7);
 \node at (1,-1)[vertex] (v8){};
 \node at (2,-2)[vertex] (v9){};
 \node at (3,-3)[vertex] (v10){};
 \node at (5,-5)[vertex] (v11){}; 
\draw (v0)--(v8)--(v9)--(v10);
\draw [dashed] (v10)--(v11);
\node at (-1.5,-1) {$P_0$};
\node at (1,-5) {$P_1$};
\node at (4,-3
) {$P_k$};

  \end{tikzpicture}
\end{center}

The tree $T$ has order $(k+1)a$, which is just one more than the amount
of vertices in Lemma \ref{t1}.
Now we show $T$ is not $A$-burnable. Otherwise, there exists
$v_1,v_2,\ldots,v_k$ such that $T$ is covered by $\cup_{i=1}^k
N_a[v_i]$.
By Pigeon-hole principle, one of the paths  $P_0,P_1,\ldots,P_k$  will
not contain $v_1,v_2,\ldots,v_k$, and the leaf vertex on this path is 
in any $N_{a-1}[v_i]$. Thus, $T$ is not $A$-burnable.
\end{proof}

 The following corollary is a slight improvement of Theorem 7 of \cite{burn4}.
\begin{corollary}
  For any connected graph $G$, $b(G)\leq 
\frac{-3+\sqrt{8n+17}}{2}\approx \sqrt{2n}-\frac{3}{2}$.
\end{corollary}
\begin{proof}
  Let $A=\{k,k-1,\cdots,1\}$. By Lemma \ref{l1}, any Tree of order
  $n\leq (\sum_{i=1}^ki) + k-1=
\frac{k^2+3k-2}{2}$ is $A$-burnable. Solving $k$ we get
  $k\leq\frac{-3+\sqrt{8n+17}}{2}$. Thus,
  $b(T)\leq\frac{-3+\sqrt{8n+17}}{2}$. By Equation \eqref{eq:3}, the
  same bound holds true for $b(G)$.
\end{proof}

\section{Proof of Theorems \ref{t1} and \ref{t2}}
We have seen that Lemma \ref{l1} is sharp when all $a_i$'s are equal.
The improvement can be made when $a_i$'s are distinct. Let $g(A)$
be a function of $A$ so that any tree $T$ with order at most $g(A)$
is $A$-burnable. In the proof of Lemma \ref{t1}, we show
that 
$$g(A)\leq g(A\setminus \{a_k\})+a_k.$$
The idea is to show a recursive bound 
$$g(A)\leq \max_{1\leq i\leq k-1}\{g(A\setminus \{a_i\})+a_i\}
+\left\lfloor\frac{k-1}{3}\right\rfloor$$
where $k$ is the number of (distinct) elements in $A$.
We first prove the following Lemma.
 
\begin{lemma}\label{l2}
  For any $k-1$ distinct positive integers $a_1<a_2<\cdots<a_{k-1}$, there
  exists an $a_i$ such that $2\lfloor\frac{k-1}{3}\rfloor\leq a_i \leq
  a_{k-1}-\lfloor\frac{k-1}{3}\rfloor$.
\end{lemma}

\begin{proof}
  Let $j=\lfloor\frac{k-1}{3}\rfloor$ and
  $A=\{a_1,a_2,\ldots,a_{k-1}\}$. Divide $[1,a_{k-1}]$ into $3$ intervals: 
$$[1,2j-1]\cup[2j,a_{k-1}-j]\cup[a_{k-1}-j+1,a_{k-1}].$$ There are at most $2j-1$ elements  of  $A$ in the first interval. There are at most $j$ elements of
  $A$ in the last interval. Since $3j-1<k-1$, there exists at least one
  element of $A$ in the middle interval. Call this element $a_i$.
\end{proof}

\begin{lemma}\label{l3}
  For all integer $k\geq 1$.
$$\sum_{i=1}^k \left\lfloor\frac{i-1}{3}\right\rfloor =\left\lfloor\frac{k^2-3k+2}{6}\right\rfloor.$$
\end{lemma}
\begin{proof}
For $k=3s$, we have
$$\sum_{i=1}^k\left\lfloor\frac{i-1}{3}\right\rfloor= 3\sum_{j=1}^s (j-1)
=\frac{3s(s-1)}{2}=\left\lfloor\frac{k^2-3k+2}{6}\right\rfloor.$$
For $k=3s+1$, we have
$$\sum_{i=1}^k\left\lfloor\frac{i-1}{3}\right\rfloor= 3\sum_{j=1}^s (j-1) +s
=\frac{3s(s-1)}{2}+s=\left\lfloor\frac{k^2-3k+2}{6}\right\rfloor.$$
For $k=3s+2$, we have
$$\sum_{i=1}^k\left\lfloor\frac{i-1}{3}\right\rfloor= 3\sum_{j=1}^s (j-1) +2s
=\frac{3s(s-1)}{2}+2s=\left\lfloor\frac{k^2-3k+2}{6}\right\rfloor.$$

\end{proof}

\begin{theorem}\label{t3}
   Let $A$ be a  set of $k$ distinct positive
  integers $a_1<a_2<\cdots<a_k$. If a tree $T$ has order at
  most $$\left(\sum_{i=1}^ka_i\right) + a_{k}-1+ \left\lfloor\frac{k^2-3k+2}{6}\right\rfloor.$$
  then $T$ is $A$-burnable.
\end{theorem}

\begin{proof}
  Let $f(k):=\lfloor\frac{k^2-3k+2}{6}\rfloor$. By Lemma \ref{l3}, we have
$f(k)=f(k-1)+ \lfloor\frac{k-1}{3}\rfloor$. Now we use induction on $k$.

Initial case $k=1$: $A=\{a_1\}$.
by Lemma \ref{l1}, if a tree $T$ has order at most
$2a_1-1$, then $T$ is $\{a_1\}$-burnable. The statement holds true for
$k=1$ since $f(1)=0$.

Now assume this statement holds true for any set of $k-1$ distinct
positive integers. Consider the case $A=\{a_1,a_2,\ldots,a_k\}$. We
need to prove that if a tree $T$ has order at most
$a_1+a_2+\cdots+2a_{k}-1+f(k)$ then $T$ is $A$-burnable.

 Let $j=\lfloor\frac{k-1}{3}\rfloor.$ By Lemma
\ref{l1}, there exists $a_i$ that satisfies $2j\leq a_i \leq
a_{k-1}-j$.
 Choose an arbitrary root $r$ and view $T$ as a rooted tree.
Let $u$ be the leaf vertex which has the
farthest distance
away from the root $r$. If $d(r,u)\leq a_k-1$, then $V(T)\subset N_{a_k-1}(r)$;
thus $T$ is $A$-burnable. So, we can assume $d(r,u)\geq a_k$.
 We will name three vertices $v_i,t,v_k$
on the $ru$-path such that $d(u,v_i)=a_i-1$, $d(u,t)=a_i-1+j$, and
$d(u,v_k)=a_{k-1}$. 
Let $T_1$ be the subtree rooted at $t$. 
There are two cases:

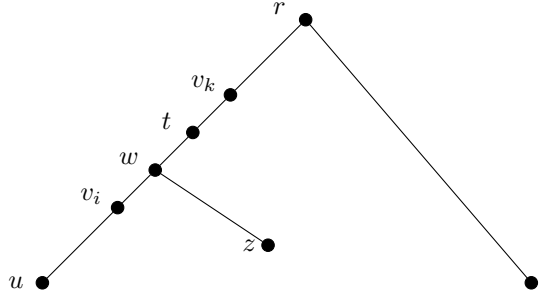
\begin{center}
  \begin{tikzpicture}
    [scale= 0.5, vertex/.style={circle,scale=0.5,
 draw=black,
      fill=black}]
   \node at (0,0)[vertex] (r) {};
\node at (-2,-2)[vertex] (vk) {};
\node at (-3,-3)[vertex] (t) {};
\node at (-4,-4)[vertex] (w) {};
\node at (-5,-5)[vertex] (v1) {};
\node at (-1,-6)[vertex] (w2) {};
\node at (-7,-7)[vertex] (u) {};
\node at (6,-7)[vertex] (r1) {};
\draw (r)--(vk)--(t)--(w)--(v1)--(u);
\draw (r)--(r1);
\draw (w)--(w2);
\node at (-0.7,0.3) {$r$};
\node at (-2.7,-1.7) {$v_k$};
\node at (-3.7,-2.7) {$t$};
\node at (-4.7,-3.7) {$w$};
\node at (-5.7,-4.7) {$v_i$};
\node at (-7.7,-7) {$u$};
\node at (-1.5,-6) {$z$};
  \end{tikzpicture}
\end{center}

\begin{description}
\item[Case 1:]
$T_1\subseteq N_{a_i-1}[v_i]$.
Let $T_2=T\setminus T_1$.
Notice $|T_1|\geq a_i+j$. Then, 
\begin{align*}
  |T_2|&\leq|T|-|T_1|\\
&\leq a_1+a_2+\cdots+2a_k-1+f(k)-(a_i+j)\\
&=  a_1+a_2+\cdots+\hat a_i+\cdots+2a_k-1+f(k-1).
\end{align*}
By inductive hypothesis, $T_2$ is $(A\setminus\{a_i\})$-burnable.
Thus, $T$ is $A$-burnable.

\item[Case 2:] $T_1\not\subseteq N_{a_i-1}[v_i]$.
Then there is a vertex $z\in T_1$ such that $d(v_i, z)\geq a_i$.
Let $w$ be the closest vertex on the path $rt$ to $z$.
Observe that $w$ is not in the subtree rooted at $v_i$. Thus, $w$ is between
$v_i$ and $t$.
We have
$$d(w,z)=d(v_i,z)-d(v_i,w)\geq a_i-d(w,v_i)\geq a_i-d(v_i,t)\geq a_i-j\geq j.$$
The last inequality uses Lemma \ref{l2} for the choice of $a_i$.
 
Let $v_k$ be a vertex on the path from $u$ to the
root with distance $d(u,v_k)$. Let $T_3$ be the subtree rooted at
$v_k$ and let $T_4:=T\setminus T_3$ be the remaining subtree. We have that $|T_3|\geq
a_{k-1}+ d(w,z)\geq a_{k-1}+j$.

\begin{align*}
  |T_4|&\leq|T|-|T_3|\\
&\leq a_1+a_2+\cdots+2a_{k}-1+f(k)-(a_{k-1}+j)\\
&=  a_1+a_2+\cdots+a_{k-2}+2a_k-1+f(k-1).
\end{align*}
By inductive hypothesis, $T_4$ is
$(A\setminus\{a_{k-1}\})$-burnable. Clearly, $T_3$ is $\{a_{k-1}\}$-burnable.
Putting together, $T$ is $A$-burnable. 
\end{description}
The inductive proof is finished.
\end{proof}
\begin{proof}{Proof of Theorem \ref{t1}}
Let $A=(1,2,\ldots,k).$ Applying Theorem \ref{t3}, any tree of $n$
vertices is $[k]$-burnable if $$n\leq
1+2+\dots+k+ k-1+\left\lfloor \frac{(k^2-3k+2)}{6} \right\rfloor
=\left\lfloor \frac{2k^2+3k-2}{3} \right\rfloor.$$ 
Note that $\left\lfloor \frac{2k^2+3k-2}{3} \right\rfloor$ equals to $\frac{2k^2+3k-3}{3}$
if $k$ is divisible by $3$; equals to $\frac{2k^2+3k-2}{3}$ otherwise.
In either case, $G$ is $[k]$-burnable if
 $n\leq \frac{2k^2+3k-3}{3}$. Solving for $k$, we have $k\geq
\frac{-3+\sqrt{24n+33}}{4}$. Since $k$ is an integer, we can take ceiling on 
the bound of $k$.
Thus for any tree $T$ of $n$ vertices, $$b(T)\leq
\left \lceil\frac{-3+\sqrt{24n+33}}{4}\right \rceil.$$ By equation
\eqref{eq:3}, the same bound holds for all connected graphs $G$.
\end{proof}

\begin{lemma}\label {l4}
  If $G$ is connected and the radius satisfies $r(G)\geq 3$, then the complement $\bar G$ is also
  connected and $r(\bar G)\leq 2$.
\end{lemma}

\begin{proof}
  Since $r(G)\geq 3$, there exists a pair of vertex $(u,v)$ with
  distance at least $3$. Let $S$ be the set of all neighbors of $v$ in the graph
  $G$. For any vertex not in $S\cup \{v\}$, it is directly connected to $v$ in
  the complement graph $\bar G$. For any vertex $x$ in $S$, both $xu$ and $uv$
are edges of $\bar G$. Thus, the complement
  graph $\bar G$ has radius at most $2$.
\end{proof}
\begin{proof}[Proof of Theorem \ref{t2}:]
By Lemma \ref{l4}, either $r(G)$ or $r(\bar G)$ is at
most $2$. Without loss of generality, we can assume $r(\bar G)\leq 2$,
which implies $b(\bar G)\leq 3$. We have the following cases.
\begin{description}
  \item[case 1]
$n\leq 4$. Since both $G$ and $\bar G$ are connected, the only graph
$G$ that can exist is the path $P_4$. In this case $G=\bar G=P_4$.
 Note, $b(P_4)=2$. This satisfies $$b(G)\cdot b(\bar G)=4 < n+4.$$

\item[case 2]
$n\geq 5$. By Theorem \ref{t1}, $ b(G_n)\leq
\left \lceil\frac{-3+\sqrt{24n+33}}{4}\right
\rceil$. $$b(G)\cdot b(\bar G) \leq 3\cdot \left \lceil\frac{-3+\sqrt{24n+33}}{4}\right
\rceil.$$ Now we show this bound is at most $n+4$.
When $n=5,6,7,8$, $\lceil\frac{-3+\sqrt{24n+33}}{4}
\rceil =3$, so $3\cdot 3=9\leq n +4$. It holds for $n=5,6,7$.

Now we assume $n\geq 9$, we use $\lceil\frac{-3+\sqrt{24n+33}}{4}
\rceil \leq \frac{-3+\sqrt{24n+33}}{4}+1$.  It is sufficient to show
$$\frac{-3+\sqrt{24n+33}}{4}+1<n+4.$$
A simple calculation yields $0< n^2-7n-8$. This true is for all $n\geq 9$.
\end{description}

From above argument, the equality holds only when $n=5$ and $b(G)=b(\bar G)=3$.
Now assume $n=5$. If $G$ contains a vertex $v$ of degree $3$ or $4$, then
$b(G)\leq 2$ since we $N[v]$ can covers at least 4 vertices. Thus
 all degrees of $G$ are at most $2$. For the same reason, all degrees of $\bar G$
are at most $2$. This implies that all degrees in $G$ and in $\bar G$ are exactly $2$.
Since both $G$ and $\bar G$ are connected and $n=5$, the only possible case is 
$G=\bar G=C_5$.
\end{proof}

% The proof above shows the following bound holds. It improved Corollary
% 21 of \cite{burn2}.
% \begin{corollary}
%   If the graph $G$ and $\bar G$ are both connected, then
%   \begin{align}
% b(G)\cdot b(\bar G)& \leq 3\cdot \left \lceil\frac{-3+\sqrt{24n+33}}{4}\right
% \rceil,\\
%  b(G)+ b(\bar G)& \leq 3+ \left \lceil\frac{-3+\sqrt{24n+33}}{4}\right
% \rceil.    
%   \end{align}
% \end{corollary}

% \begin{theorem}
%   Conjecture 1 holds for all connected graphs $G$ of order at most 21.
% \end{theorem}
% \begin{proof}
% It is sufficient to show that it holds for any trees.
%   Applying Lemma \ref{l2} to $A=[k]$, for $k=1,2,3,4,5$, we have
%   \begin{enumerate}
%   \item Any tree of order at most $1$ is $[1]$-burnable.
%   \item Any tree of order at most $4$ is $[2]$-burnable.
%  \item Any tree of order at most $8$ is $[3]$-burnable.
%  \item Any tree of order at most $14$ is $[4]$-burnable.
%  \item Any tree of order at most $21$ is $[5]$-burnable.
%   \end{enumerate}
% To verify Conjecture 1 for trees of order at most $21$, it is sufficient to show
% \begin{enumerate}
% \item Any tree of order $9$ is $[3]$-burnable.
% \item Any tree of order $15$ is $[4]$-burnable.
% \item Any tree of order $16$ is $[4]$-burnable.
% \end{enumerate}
% Suppose that $T$ is a tree of order $9$. Consider the diameter $D(T)$.
% If $D(T)\leq 5$, then $T$ is $\{3\}$-burnable. Suppose $D(T)\geq 6$.

% \end{proof}

\end{document}